\documentclass[preprint,12pt]{elsarticle}

\usepackage{amssymb}
\usepackage{graphicx}
\usepackage{mathptmx}      

\usepackage{amssymb,amsfonts,amsmath,amscd,latexsym}
\usepackage[english]{babel}
\usepackage{geometry}
\usepackage{latexsym}
\usepackage{amssymb}
\usepackage[all]{xy}
\usepackage{amssymb}
\usepackage{latexsym}
\usepackage{graphicx}
\usepackage{appendix}
\usepackage[active]{srcltx}
\usepackage{srcltx}
\usepackage{amsthm}

\theoremstyle{theorem}
\newtheorem{theorem}{Theorem}
\newtheorem{lemma}{Lemma}
\newtheorem{corollary}{Corollary}
\theoremstyle{definition}

\newtheorem{application}{Application}

\begin{document}

\begin{frontmatter}



\title{Relations between ranks of matrix polynomials}

\author{Vasile Pop}
\ead{vasile.pop@math.utcluj.ro}
\address{Technical University of Cluj-Napoca,  Str. C. Daicoviciu 15, 400020, Cluj-Napoca, Romania}

\cortext[cor1]{Corresponding author} 
\begin{abstract}
We show that the sum of ranks of two matrix polynomials is the same as the sum of the rank of the matrix obtained by applying the greatest common divisor of the
polynomials, with the rank of the matrix obtained  by applying the lowest common multiple of the polynomials. Many applications, for older or more recent problems, of this result are obtained.
\end{abstract}


\end{frontmatter}


\section{Introduction.}\label{sec1}
Let $(K,+,\cdot)$ be a field, let $K[X]$ be the ring of polynomials with
coefficients in $K$ and $\mathcal{M}_{n}(K)$ the ring of square matrices of
order $n\ge2$ with coefficients in $K$. Let $D:=(f,g)$ denote the greatest
common divisor and $M:=[f,g]$ denote the lowest common multiple of the
polynomials $f,g$.
The main result of this paper is  following theorem.

\begin{theorem}[Sum of ranks of matrix polynomials]\label{thm1}
For any two polynomials $f,g\in K[X]$ and for any
matrix $A\in\mathcal{M}_{n}(K)$ the following relation holds
\[
\mathrm{rank}\, f(A)+\mathrm{rank}\, g(A)=\mathrm{rank}\, D(A)+\mathrm{rank}\,
M(A).
\]
\end{theorem}
From the above relation one can obtain other interesting relations, particular
theorems that characterize some matrices as idempotent, tripotent, involutive
matrices and other results which may  be well known for experts. These are exposed in Section \ref{sec3} where we reviewed some particular results obtained in \cite{1}, \cite{2} and \cite{3}.
\section{Proof of Theorem \ref{thm1}.}\label{sec2}
We prove the relation using the method of elementary
transformations in block matrices which appears heavily in \cite{6}. Other sources of methods for matrices theory and applications which we are using are from \cite{4}, \cite{5} and \cite{7}. The main result that we use is the following well-known lemma:
\begin{lemma}If $f,g\in K[X]$ then there are $\varphi_{1}%
,\varphi_{2}\in K[X]$ such that
\begin{equation}\label{eqno(1)}
f\cdot\varphi_{1}+\varphi_{2}\cdot g=D.
\end{equation}

\end{lemma}
\begin{proof} (\textbf{of Theorem \ref{thm1}.})
For the matrix $A$ the relation (\ref{eqno(1)}) become
\begin{equation}\label{eqno(2)}
f(A)\cdot\varphi_{1}(A)+\varphi_{2}(A)\cdot g(A)=D\left(  A\right)  
\end{equation}

We start with the block matrix $B\in\mathcal{M}_{2n}(K)$
\[
B=\left[
\begin{array}
[c]{c|c}%
f(A) & 0\\\hline
0 & g(A)
\end{array}
\right]
\]
with $\mathrm{rank}\,B=\mathrm{rank}\,f(A)+\mathrm{rank}\,g(A)$ and we perform
the following elementary transformations:
\[
\left[
\begin{array}
[c]{c|c}%
f(A) & 0\\\hline
0 & g(A)
\end{array}
\right]  \overset{C_{1}}{\longrightarrow}\left[
\begin{array}
[c]{c|c}%
f(A) & f(A)\cdot\varphi_{1}(A)\\\hline
0 & g(A)
\end{array}
\right]  \overset{L_{1}}{\longrightarrow}%
\]%
\[
\left[
\begin{array}
[c]{c|c}%
f(A) & D(A)\\\hline
0 & g(A)
\end{array}
\right]  \overset{L_{2}}{\longrightarrow}\left[
\begin{array}
[c]{c|c}%
f(A) & D(A)\\\hline
-M(A) & 0
\end{array}
\right]  \overset{C_{2}}{\longrightarrow}%
\]%
\[
\left[
\begin{array}
[c]{c|c}%
0 & D(A)\\\hline
-M(A) & 0
\end{array}
\right]  =C,
\]
with $\mathrm{rank}\,C=\mathrm{rank}\,D(A)+\mathrm{rank}\,M(A)$.

The matrices used for the elementary transformations on columns $(C_{1}%
,C_{2})$ and on rows $(L_{1},L_{2})$ are
\[
C_{1}= \left[
\begin{array}
[c]{c|c}%
I_{n} & \varphi_{1}(A)\\\hline
0 & I_{n}%
\end{array}
\right]  ,\quad L_{1}= \left[
\begin{array}
[c]{c|c}%
I_{n} & \varphi_{2}(A)\\\hline
0 & I_{n}%
\end{array}
\right]  ,
\]
\[
C_{2}= \left[
\begin{array}
[c]{c|c}%
I_{n} & 0\\\hline
-\psi_{2}(A) & I_{n}%
\end{array}
\right]  ,\quad L_{2}= \left[
\begin{array}
[c]{c|c}%
I_{n} & 0\\\hline
-\psi_{1}(A) & I_{n}%
\end{array}
\right],
\]
where $g(A)=\psi_{1}(A)\cdot D(A)$ and $f(A)=D(A)\cdot\psi_{2}(A)$. We used relation (\ref{eqno(2)}) in the third step and $\psi_1,\psi_2$ are  polynomials obtained by dividing the given polynomials $f,g$ to their greatest common divisor,

We obtain the relations $$C=L_{2} \cdot L_{1} \cdot B \cdot C_{1} \cdot C_{2}, \ \ M\cdot D= f\cdot g,$$
which concludes the proof.
\end{proof}
\section{Applications.}\label{sec3}

\begin{corollary}\label{cor1}$M(A)=0$ if and only if
$\mathrm{rank}\, f(A)+\mathrm{rank}\, g(A)=\mathrm{rank}\, D(A).$
\end{corollary}

\begin{corollary}\label{cor1'} If $m_{A}$ is the minimal polynomial of the
matrix $A$ then $m_{A}|M$ if and only if $\mathrm{rank}\, f(A)+\mathrm{rank}\, g(A)=\mathrm{rank}\, D(A).$
\end{corollary}

\begin{corollary}\label{cor2} The polynomials $f$ and $g$ are coprime
polynomials if and only if for any matrix $A\in\mathcal{M}_{n}(K)$ the
following relation hold
\[
\mathrm{rank}\, f(A)\cdot g(A)+n=\mathrm{rank}\, f(A)+\mathrm{rank}\, g(A).
\]
\end{corollary}
The following corollary appears in \cite{3}.
\begin{corollary}\label{cor3} If $f$ and $g$ are coprime polynomials, $\ \ $then
$\ \ \ \ f(A)\cdot g(A)=0$
if and only if $\mathrm{rank}\,f(A)+\mathrm{rank}\,g(A)=n.$
\end{corollary}

\begin{application} (\cite{1}, \cite{2})The matrix $A\in\mathcal{M}_{n}(K)$ is
idempotent $(A^{2}=A)$ if and only if
\[
\mathrm{rank}\, A+\mathrm{rank}\, (I_{n}-A)=n.
\]
\end{application}

\begin{proof} Take the polynomials $f(x)=x$, $g(x)=1-x$ and then we have
$ D(x)=1,\quad M(x)=x-x^{2}. $
Now we apply Corollary \ref{cor1} or Corollary \ref{cor3}.
\end{proof}
\begin{application} (\cite{1}, \cite{2}) The matrix $A\in\mathcal{M}_{n}(K)$ is
involutive $(A^{2}=I_{n})$ (the field $K$ of characteristic different than 2) if and
only if
\[
\mathrm{rank}\, (I_{n}-A)+\mathrm{rank}\, (I_{n}+A)=n.
\]
\end{application}

\begin{proof} We apply Corollary \ref{cor3} for the polynomials $f(x)=1-x,\quad g(x)=1+x.$
\end{proof}

\begin{application} (\cite{2}) Using the hypothesis that the field $K$ does not have
the characteristic 2, the following statements are equivalent:
\begin{itemize}
\item[1)] The matrix $A$ is tripotent $(A^{3}=A)$.

\item[2)] $\mathrm{rank}\, A+\mathrm{rank}\, (I_{n}-A^{2})=n$.

\item[3)] $\mathrm{rank}\, (I_{n}-A)+\mathrm{rank}\, (A+A^{2})=n$.

\item[4)] $\mathrm{rank}\, A+\mathrm{rank}\, (I_{n}-A)+\mathrm{rank}\, (I_{n}+A)=2n$.
\end{itemize}

\end{application}
\begin{proof} 

$"1)\Leftrightarrow2)"$. We apply Corollary \ref{cor3} for the
polynomials $f(x)=x,\quad g(x)=1-x^{2}.$

$"1)\Leftrightarrow3")$. We apply Corollary 3 for the polynomials
$ f(x)=1-x,\quad g(x)=x+x^{2}.$

$"2)\Leftrightarrow4)"$. It is enough to prove the relation
\[
n+\mathrm{rank}\, (I_{n}-A^{2})=\mathrm{rank}\, (I_{n}-A)+\mathrm{rank}\,
(I_{n}+A),
\]
which follows from Corollary \ref{cor2} for $f(x)=1-x$ and $g(x)=1+x$.
\end{proof}
\begin{application} Let $A\in\mathcal{M}_{n}(K)$ and $f_{A}\in K[X]$ be its
characteristic polynomial, which we decompose in a product of polynomials that
are coprime two by two
\[
f_{A}=f_{1}\cdot f_{2}\cdot\ldots\cdot f_{k}.
\]
Then we have the relation
\[
\mathrm{rank}\, f_{1}(A)+\mathrm{rank}\, f_{2}(A)+\ldots+\mathrm{rank}\,
f_{k}(A)=(k-1)n.
\]
\end{application}

\begin{proof} We apply Corollary \ref{cor3} for $f=f_{1}\cdot f_{2}\cdot\ldots\cdot
f_{k-1}$, $g=f_{k}$ and we obtain
\[
\mathrm{rank}\, f(A)+\mathrm{rank}\, f_{k}(A)=n.
\]
We apply now Corollary \ref{cor2} for $f=f_{1}\cdot f_{2}\cdot\ldots\cdot f_{k-2}$,
$g=f_{k-1}$ and we obtain
\[
\mathrm{rank}\, (f_{1}\cdot f_{2}\cdot\ldots\cdot f_{k-2}) +\mathrm{rank}\,
f_{k-1}(A)+\mathrm{rank}\, f_{k}(A)=2n.
\]
By induction, we obtain the result.
\end{proof}
\begin{application} For any matrix $A\in\mathcal{M}_{n}(K)$ the following
relation is verified
\[
\mathrm{rank}\, (A+A^{2})+\mathrm{rank}\, (A-A^{2})=\mathrm{rank}\,
A+\mathrm{rank}\, (A-A^{3}).
\]
\end{application}

\begin{proof} We apply Theorem \ref{thm1} with
\[
f(x)=x+x^{2},\quad g(x)=x-x^{2},\quad D(x)=x,\quad M(x)=x-x^{3}.
\]
\end{proof}

\begin{application} (\cite{2}) For a matrix $A\in\mathcal{M}_{n}(K)$ the following
statements are equivalent:
\begin{itemize}
\item[1)] $A^{3}=A^{5}.$

\item[2)] $\mathrm{rank}\, A^{3}+\mathrm{rank}\, (I_{n}-A^{2})=n.$

\item[3)] $\mathrm{rank}\, (I_{n}-A)+\mathrm{rank}\, (A^{3}+A^{4})=n.$

\item[4)] $\mathrm{rank}\, (I_{n}+A)+\mathrm{rank}\, (A^{3}-A^{4})=n.$

\item[5)] $\mathrm{rank}\, A^{3}+\mathrm{rank}\, (I_{n}-A)+\mathrm{rank}\,
(I_{n}+A)=2n.$

\item[6)] $\mathrm{rank}\, (A-A^{2})+\mathrm{rank}\, (A^{3}+A^{4})=\mathrm{rank}\, A.$

\item[7)] $\mathrm{rank}\, (A+A^{2})+\mathrm{rank}\, (A^{3}-A^{4})=\mathrm{rank}\, A.$

\item[8)] $\mathrm{rank}\, (A^{3}+A^{4})+\mathrm{rank}\, (A^{3}-A^{4})=\mathrm{rank}%
\, A^{3}$.
\end{itemize}

\end{application}
 

\begin{thebibliography}{2}
\bibitem{1} V.  Pop, The characterization of same linear maps using
the rank, \textit{Journal of Science and Arts (JOSA)}, \textbf{3} (36) (2016) , 225--228.
\bibitem{2} V. Pop,, The method of elementary operations in block
matrices for the determination of the annihilator polynomials of some
matrices, \textit{Journal of Science and Arts (JOSA)}, \textbf{4} (37) (2016), 295--302.
\bibitem{3} V. Pop, The decomposition of some annihilator
polynomials of linear maps in coprime polynomials, \textit{Journal of Science and Arts (JOSA)}, \textbf{4} (41) (2017), 647--650.
\bibitem{4} I. Tyan, G.  Styan, Linear Algebra and its Applications \textbf{135, 101} (2001).
\bibitem{5} I. Tyan, G.  Styan, Journal of Computational and Applied
Mathematics, \textbf{191, 77} (2006).
\bibitem{6} F. Zhang , Matrix Theory -- Basic Results and Techniques,
Springer, New York 1999.
\bibitem{7} D. S. Bernstein, Matrix Mathematics. Theory, Facts and
Formulas, Princeton Univ. Press 2009.
\end{thebibliography}
\end{document}